\documentclass[12pt]{amsart}

\usepackage{amssymb,amsfonts,amsthm,amsmath,amscd}

\theoremstyle{plain}
\newtheorem{proposition}{Proposition}%[section]
\newtheorem{lemma}[proposition]{Lemma}
\newtheorem{theorem}[proposition]{Theorem}

\newtheorem{conjecture}[proposition]{Conjecture}
\newtheorem{question}[proposition]{Question}

\theoremstyle{definition}
\newtheorem*{ack}{Acknowledgement}

\theoremstyle{remark}
\newtheorem{remark}{Remark}

\def\N{\mathbb{N}}
\def\NP{\mathcal{N}}
\def\P{\mathcal{P}}

\def\G{\mathcal{G}}
\def\R{\mathcal{R}}
\def\E{\mathcal{E}}

\def\A{\mathcal{A}}
\def\B{\mathcal{B}}
\def\C{\mathcal{C}}
\def\D{\mathcal{D}}
\def\F{\mathcal{F}}
\def\Q{\mathcal{Q}}

\begin{document}

\title{Two variants of Wythoff's game preserving its $\P$-positions}

\author{Nhan Bao Ho}
\address{Department of Mathematics, La Trobe University, Melbourne, Australia 3086}
\email{nbho@students.latrobe.edu.au, honhanbao@yahoo.com}


\begin{abstract}
We present two variants of Wythoff's game. The first game is a restriction of Wythoff's game in which removing tokens from the smaller pile is not allowed if the two entries are not equal. The second game is an extension of Wythoff's game obtained by adjoining a move allowing players to remove $k$ tokens from the smaller pile and $l$ tokens from the other pile provided $l < k$. We show that both games preserve the $\P$-positions of Wythoff's game. This resolves a question raised by Duch{\^e}ne,  Fraenkel, Nowakowski and Rigo. We give formulas for those positions which have Sprague-Grundy value 1. We also prove several results on the Sprague-Grundy functions.
\end{abstract}

%====================================================================================================================
%====================================================================================================================
%====================================================================================================================
\maketitle

\section{Introduction}
Wythoff's game, introduced by Willem Abraham Wythoff \cite{Wyt}, is a variant of Nim involving two piles of tokens. Two players move alternately. In each move, one can either remove an arbitrary number of tokens from one pile as in Nim or remove an arbitrary equal numbers of tokens from both piles. The game ends when the two piles become empty. The player who makes the last move wins. The position with the two piles of $a$ and $b$ tokens is denoted by $(a,b)$ which is also identical to $(b,a)$ because of symmetry.  A position is called a \emph{winning position} (known as \emph{$\NP$-position}) if the player about to move from there has a plan of moves to wins. Otherwise, it is a \emph{losing position} (known as \emph{$\P$-position}). Wythoff  showed that $(a,b)$ is a losing position if and only if $a = \lfloor \phi n \rfloor, b = \lfloor \phi^2 n \rfloor$ for some integer $n$, where $\phi = (1+\sqrt{5})/2$ and $\lfloor . \rfloor$ denotes the integer part. Aspects of Wythoff's game are discussed in \cite{blass, Dress, landman, nivasch}. Some variants involving more than two piles of tokens can be found at \cite{Ext-Res, Geo-ext, Nimhoff, End-Wyt, anew}.

Many natural variants of Wythoff's game involve either {\em restrictions}, where some moves of Wythoff's game are eliminated \cite{Gen-Connell, Nim-Wythoff, Ho}, or {\em extensions}, where certain additional moves are permitted \cite{Heapgame, Howtobeat, Gen-Fra, Adjoining, Ho, hog, Some-Gen}. Duch\^{e}ne \emph{et al.}~\cite{Ext-Res} examined restrictions and extensions under the added assumption that the moves in question be ``playable from any game position'' (such games are said to be ``invariant" in \cite{inv}); as an example of a move which is not of this type, they offered the following: remove an odd number of tokens from a position $(a,b)$ if $a$ or $b$ is a prime number, and an even number of tokens otherwise. With this definition of restriction, Duch\^{e}ne \emph{et al.}~\cite{Ext-Res} proved that there is no restriction of Wythoff's game preserving its $\P$-positions. Furthermore, they asked if there exists a variant of Wythoff's game preserving its $\P$-positions  which is not an extension in the sense of their paper \cite[Question 1]{Ext-Res}.  This paper presents two such variants; one is a restriction and one is an extension, in the general sense of these terms.

Let $S$ be a finite set of nonnegative integers. The smallest nonnegative integer not in $S$ is called the \emph{minimum excluded number} of $S$, denoted by $mex(S)$. For a given game, if there exists a move from $p$ to $q$, then $q$ is called a \emph{follower} of $p$. The \emph{Sprague-Grundy function} of a game $G$ is the function $\mathcal{G}$ from the set of positions of $G$ into the nonnegative integers defined inductively by
\[\mathcal{G}(p) = mex\{\mathcal{G}(q)| q \text{ is a follower of } p\}\]
with $mex\{\} = 0$. The value $\mathcal{G}(p)$ is called the \emph{Sprague-Grundy value} at $p$.

The outline of this paper is as follows. In the next section, we study a 2-pile variant of Wythoff's game that we call \emph{$\R$-Wythoff}. Each move is either to remove a positive number of tokens from the larger pile (or any pile if the two piles are the same size) or to remove the same number of tokens from both piles. Note that if the sizes of the two piles are not equal, then removing tokens from the smaller pile is not allowed. This is therefore a restriction of Wythoff's game. We show that $\R$-Wythoff preserves the $\P$-positions of Wythoff's game. Moreover, we prove that there is no restriction of $\R$-Wythoff preserving its $\P$-position. We describe those positions which have Sprague-Grundy value 1. We then investigate some properties of the Sprague-Grundy function, which is denoted by $\G_{\R}$.

In Section 3, we present an extension of Wythoff's game obtained by adjoining a move removing $k$ tokens from the smaller pile (or any pile if the two piles have the same size) and $l$ tokens from the other pile where $l < k$. We call this extension  \emph{$\E$-Wythoff}. We show that $\E$-Wythoff  also preserves $\P$-positions of Wythoff's game. We give formulas for those positions which have Sprague-Grundy value 1 in $\E$-Wythoff before proving several results for the Sprague-Grundy function, which is denoted by $\G_{\E}$.

This paper is a continuation of our work on 2-pile variants of Nim \cite{MEuclid, Min, CHL, Ho}. In particular, in \cite{Ho}, we examine several variants of Wythoff's game whose $\P$-positions are obtained by adding 1 to each entry of $\P$-positions of Wythoff's game.

%====================================================================================================================
%====================================================================================================================
%====================================================================================================================
\medskip

\section{$\R$-Wythoff}

Let $\phi = (1+\sqrt{5})/2$. Then $\phi^2 = \phi + 1$. Therefore, for every positive integer $n$, we have
\[\lfloor \phi^2n \rfloor = \lfloor \phi n + n \rfloor = \lfloor \phi n \rfloor + n.\]
The following lemma shows that the two sets $\{a_i | i \geq 1\}$ and $\{b_i | i \geq 1\}$, where $a_i = \lfloor \phi i \rfloor$ and $b_i = \lfloor \phi^2 i \rfloor$, are complementary. That is
\begin{align*}
\begin{cases}
\{a_i | i \geq 1\} \cup \{b_i | i \geq 1\} = \N,\\
\{a_i | i \geq 1\} \cap \{b_i | i \geq 1\} = \emptyset,
\end{cases}
\end{align*}
in which $\N$ is the set of positive integers.

%====================================================================================================================
\smallskip
\begin{lemma} \label{Comp} \cite{beatty1}
Let $a$ be a positive integer. There exists exactly one $n$ such that either $a = \lfloor \phi n \rfloor$ or $a = \lfloor \phi n \rfloor + n$. Moreover, the number $a$ cannot be of both forms.
\end{lemma}

We now show that a winning strategy in Wythoff's game can be applied to $\R$-Wythoff.
%====================================================================================================================
\smallskip
\begin{theorem} \label{BW-P}
The $\P$-positions of $\R$-Wythoff  are identical to those of Wythoff's game.
\end{theorem}

\begin{proof}
Let $\A = \{(\lfloor\phi n\rfloor, \lfloor \phi n \rfloor + n) | n \geq 0\}$. We need to show that the following two properties hold for $\R$-Wythoff:
\begin{itemize}
\item [(i)]  Every move from a position in $\A$ cannot terminate in $\A$,
\item [(ii)] From every position not in $\A$, there is a move terminating in $\A$.
\end{itemize}

For (i), note that $\A$ is the set of $\P$-positions of Wythoff's game \cite{Wyt}. Moreover, a move in $\R$-Wythoff is also legal in Wythoff's game. Since (i) holds for Wythoff's game, (i) also holds for $\R$-Wythoff.

For (ii), we can assume that $a < b$ as if $a = b$ then one can move from $(a,a)$ to $(0,0) \in \A$. By Lemma \ref{Comp}, either $a = \lfloor\phi n\rfloor$ or $a = \lfloor\phi n\rfloor + n$ for some $n$. Assume that $a = \lfloor\phi n\rfloor$. Then $b = \lfloor\phi n\rfloor + i$ for some $i \geq 1$ and $i \neq n$. If $i < n$, we have $\lfloor\phi i\rfloor < \lfloor\phi n\rfloor$. Removing $\lfloor\phi n\rfloor - \lfloor\phi i\rfloor$ tokens from both piles leads $(a,b)$ to $(\lfloor\phi i\rfloor, \lfloor\phi i\rfloor+i) \in \A$. If $i > n$, one can move from $(a,b)$ to $(\lfloor\phi n\rfloor, \lfloor\phi n\rfloor+n)$ by removing $i-n$ tokens from the larger pile. Assume that $a = \lfloor\phi n\rfloor + n$. Then $b = \lfloor\phi n\rfloor + n +i$ for some $i \geq 1$. One can move from $(a,b)$ to $(\lfloor\phi n\rfloor, \lfloor\phi n\rfloor + n)$ by removing $n+i$ tokens from the larger pile.
\end{proof}

Duch\^{e}ne \emph{et al.}~\cite{Ext-Res} defined a \emph{redundant} move of an impartial game to be a move $\mathcal{M}$ in which the set of $\P$-positions of the game is unchanged if the move $\mathcal{M}$ is eliminated. Note that a move $\mathcal{M}$ is not redundant if there exists a position $p$ such that $\mathcal{M}$ is the unique winning move from $p$. As shown in Theorem \ref{BW-P}, $\R$-Wythoff is obtained from Wythoff's game by eliminating  redundant moves.

%====================================================================================================================
\medskip
We next show that $\R$-Wythoff does not have redundant moves.

%====================================================================================================================
\smallskip
\begin{theorem} \label{no-rest-BW}
There is no restriction of $\R$-Wythoff preserving its $\P$-positions.
\end{theorem}

\begin{proof}
We will show that neither of the moves in $\R$-Wythoff is redundant. We need to show that for every positive integer $k$, the following two properties hold:
\begin{itemize}
\item [(i)]  There exists a winning position $(a,b)$ with $a < b$ such that removing $k$ tokens from the larger pile is the unique winning move.
\item [(ii)] There exists a winning position such that removing $k$ tokens from both piles is the unique winning move.
\end{itemize}

For (i), let $a = 1, b = 2+k$. Then $(a,b)$ is an $\NP$-position (i.e., a winning position). Moreover, the move removing $k$ tokens from the larger pile is the unique winning move. In fact, the other type of move is to remove 1 token from both piles leading $(a,b)$ to (0,1+k) which is an $\NP$-position.

For (ii), we first claim that there exist positive integers $n,m$ such that $\lfloor\phi n\rfloor + k = \lfloor\phi m\rfloor$. In fact, set $n_1 = \lfloor 2\phi\rfloor = 3$, $n_2 = \lfloor3\phi \rfloor = 4$, $m_1 = 3+k$, and $m_2 = 4 + k$. We show that either $m_1$ or $m_2$ is of the form $\lfloor\phi m\rfloor$ for some $m$. Assume by contradiction that neither $m_1$ nor $m_2$ is of the form $\lfloor\phi m\rfloor$. By Lemma \ref{Comp}, there exist $r_1 < r_2$ such that $m_1 = \lfloor\phi r_1\rfloor +r_1$, $m_2 = \lfloor\phi r_2\rfloor +r_2$. Note that $\lfloor\phi r_1\rfloor < \lfloor\phi r_2\rfloor$ and so
\[1 = m_2 - m_1  = \lfloor\phi r_2\rfloor +r_2 - (\lfloor\phi r_1\rfloor +r_1)
                 = \lfloor\phi r_2\rfloor - \lfloor\phi r_1\rfloor + r_2 - r_1 \geq 2 \]
giving a contradiction. Now, if $m_1 = \lfloor\phi m\rfloor$ (resp. $m_2 = \lfloor\phi m\rfloor$), let $n = 2$ (resp. $n = 3$). Then $n, m$ satisfy the condition $\lfloor\phi n\rfloor + k = \lfloor\phi m\rfloor$. Let $a = \lfloor\phi n\rfloor + k$, $b = \lfloor\phi n\rfloor +n + k$. Then $(a,b)$ is an $\NP$-position and removing $k$ tokens from both piles is a winning move. It remains to show that this is the unique winning move. Assume by contradiction that there exists another winning move from $(a,b)$. This move must take some $l$ tokens from the larger pile leading $(a,b)$ to some position $(\lfloor\phi r\rfloor, \lfloor\phi r\rfloor + r)$. First consider the case $b-l = \lfloor\phi r\rfloor$, $a = \lfloor\phi r\rfloor + r$. We have shown the existence of $m$ such that $\lfloor\phi m\rfloor = \lfloor\phi n\rfloor + k = a$ and so $\lfloor\phi m\rfloor = \lfloor\phi r\rfloor + r$. However, this equality cannot occur by Lemma \ref{Comp}. Now consider the case $a = \lfloor\phi r\rfloor$, $b-l = \lfloor\phi r\rfloor + r$. We have
\begin{align*}
\begin{cases}
a = \lfloor\phi n\rfloor + k = \lfloor\phi r\rfloor,\\
b-l = \lfloor\phi n\rfloor + n + k - l = \lfloor\phi r\rfloor +r.
\end{cases}
\end{align*}
The first equation implies $n < r$. By substituting $\lfloor\phi n\rfloor + k$ from the first equation into the second one, we get $\lfloor\phi r\rfloor + n - l = \lfloor\phi r\rfloor + r$ which implies $n = l+r > r$ giving a contradiction. Therefore, this case is impossible.
\end{proof}

%==================================================================================

\begin{table}[ht]
\begin{center}
\begin{tabular}{c|cccccccccc}
9     &9&9&9 &5 &9 &1 &9 &5 &9 &10 \\
8     &8&8&8 &8 &8 &2 &8 &6 &7 &9 \\
7     &7&7&7 &7 &0 &7 &7 &8 &6 &5 \\
6     &6&6&6 &1 &1 &4 &5 &7 &8 &9 \\
5     &5&5&5 &0 &5 &6 &4 &7 &2 &1 \\
4     &4&4&4 &2 &3 &5 &1 &0 &8 &9\\
3     &3&3&3 &4 &2 &0 &1 &7 &8 &5 \\
2     &2&0&1 &3 &4 &5 &6 &7 &8 &9 \\
1     &1&2&0 &3 &4 &5 &6 &7 &8 &9 \\
0     &0&1&2 &3 &4 &5 &6 &7 &8 &9 \\
\hline
a/b &0&1&2&3&4&5&6&7&8&9
 \end{tabular}
\caption{Sprague-Grundy values $\G_{\R}(a,b)$ for $a,b\leq 9$}\label{T1}
\end{center}
\end{table}

%====================================================================================================================
Table \ref{T1} gives the  Sprague-Grundy values of position $(a,b)$ for $a,b\leq 9$. We now determine the positions of Sprague-Grundy value 1.

%====================================================================================================================
\smallskip
\begin{theorem} \label{BW-V1}
In $\R$-Wythoff, the position $(a,b)$ with $a \leq b$ has Sprague-Grundy value 1 if and only if $(a,b)$ is an element of the set
\[\B = \{(2,2), (4,6), (\lfloor \phi n \rfloor - 1, \lfloor \phi n \rfloor +n-1) | n \geq 1, n \neq 2\}. \]
\end{theorem}

%====================================================================================================================
\smallskip
Before proving Theorem \ref{BW-V1}, we need some lemmas. %====================================================================================================================
\smallskip
\begin{lemma} \label{A-B}
For all $m \neq n$, we have
\[(\lfloor \phi m \rfloor, \lfloor \phi m \rfloor +m) \neq (\lfloor \phi n \rfloor - 1, \lfloor \phi n \rfloor +n-1).\]
\end{lemma}
\begin{proof}
Assume by contradiction that there exist nonnegative integers $m \neq n$ such that
\begin{align*}
\begin{cases}
\lfloor \phi m \rfloor = \lfloor \phi n \rfloor - 1,\\
\lfloor \phi m \rfloor + m = \lfloor \phi n \rfloor +n-1.
\end{cases}
\end{align*}
The first equation implies $m < n$. By substituting $\lfloor \phi m \rfloor$ from the first equation into the second one and then simplifying it, we get $m = n$, giving a contradiction.
\end{proof}

%====================================================================================================================
\smallskip
\begin{lemma} \label{A-B1}
Set
\[\C = \{(\lfloor \phi n \rfloor - 1, \lfloor \phi n \rfloor +n-1) | n \geq 1  \}\]
and let $(a,b) \in \C$ with $a \leq b$. For $x \geq 1$, $y \geq 0$ such that $x \geq y$, the following two conditions hold:
\begin{itemize}
\item [(i)] $(a, b-x) \notin \C$,
\item [(ii)] $(a-x, b-y) \notin \C$
\end{itemize}
\end{lemma}

\begin{proof}
Assume that $a = \lfloor \phi n \rfloor - 1$, $b = \lfloor \phi n \rfloor +n-1$ for some $n$.

For (i), assume by contradiction that there exists $x \geq 1$ such that $(a, b-x) \in \C$. Then, there exists $m < n$ such that either
\begin{align*}
\begin{cases}
a = \lfloor \phi n \rfloor - 1 = \lfloor \phi m \rfloor - 1,\\
b-x = \lfloor \phi n \rfloor +n-1 - x = \lfloor \phi m \rfloor +m-1
\end{cases}
\end{align*}
or
\begin{align*}
\begin{cases}
b-x = \lfloor \phi n \rfloor +n - 1-x = \lfloor \phi m \rfloor - 1,\\
a = \lfloor \phi n \rfloor -1 = \lfloor \phi m \rfloor +m-1.
\end{cases}
\end{align*}
The first equation of the former case implies $m = n$ giving a contradiction. The second equation of the latter case implies $\lfloor \phi n \rfloor = \lfloor \phi m \rfloor +m$ contradicting Lemma \ref{Comp}. Therefore, (i) holds.

For (ii), assume by contradiction that there exist $x, y$ with $x \geq 1$, $y \geq 0$, $x \geq y$ such that $(a-x, b-y) \in \C$. Note that $a-x < b-y$. Then there exists $m < n$ satisfying
\begin{align*}
\begin{cases}
a-x = \lfloor \phi m \rfloor - 1,\\
b-y = \lfloor \phi m \rfloor +m-1
\end{cases}
&\Rightarrow
\begin{cases}
\lfloor \phi n \rfloor - 1 - x = \lfloor \phi m \rfloor - 1,\\
\lfloor \phi n \rfloor +n-1 - y = \lfloor \phi m \rfloor +m-1
\end{cases}\\
&\Rightarrow
\begin{cases}
\lfloor \phi n \rfloor - \lfloor \phi m \rfloor  = x,\\
\lfloor \phi n \rfloor - \lfloor \phi m \rfloor = y + m - n.
\end{cases}
\end{align*}
It follows that $x = y + m - n \leq y-1$ as $m < n$. This is a contradiction. Therefore, (ii) holds.
\end{proof}

%====================================================================================================================
\smallskip
\begin{proof}[Proof of Theorem \ref{BW-V1}]
Recall that the set of $\P$-positions of the game is
\[\P = \{(\lfloor \phi m \rfloor, \lfloor \phi m \rfloor + m)| m \geq 0\}.\]

By the definition of Sprague-grundy function, a position $p$ has Sprague-Grundy value $k > 0$ if and only if the following two conditions hold
\begin{itemize}
\item if there exists a move from $p$ to some $q$, then $\G(p) \neq \G(q)$;
\item for every $l < k$, there exists a follower $q$ of $p$ such that $\G(q) = l$.
\end{itemize}
Therefore, we need to prove that
\begin{itemize}
\item [(i)] $\B \cap \P = \emptyset$;
\item [(ii)] there is no move from a position in $\B$ to a position in $\B$;
\item [(iii)] from every position not in $\B \cup\P$, there exists a move to some position in $\B$.
\end{itemize}

For (i), since $(2,2), (4,6) \notin \P$, it is sufficient to show that
\[(\lfloor \phi m \rfloor, \lfloor \phi m \rfloor + m) \neq (\lfloor \phi n \rfloor - 1, \lfloor \phi n \rfloor +n-1)\] for all $m,n$.  This is true by Lemma \ref{A-B}.

For (ii), let $p = (a,b) \in \B$. We show that for all $x, y > 0$, the following two properties hold
\begin{subequations}
\begin{equation} \label{sub-B-x}
(a-x, b-x) \notin \B,
\end{equation}
\begin{equation} \label{sub-B-y}
(a,b-y) \notin \B.
\end{equation}
\end{subequations}
It can be checked manually that (\ref{sub-B-x}) and (\ref{sub-B-y}) hold if $(a,b)$ is either $(2,2)$ or $(4,6)$. Assume now that $a = \lfloor \phi n \rfloor - 1$, $b = \lfloor \phi n \rfloor +n-1$ for some $n$. We first show that (\ref{sub-B-x}) holds. Assume by contradiction that there exists $x > 0$ such that $(a-x, b-x) \in \B$. Then, by Lemma \ref{A-B1}, either $(a-x, b-x) = (2,2)$ or $(a-x, b-x) = (4,6)$. The first case cannot occur as $a < b$. The second case occurs if and only if $b-a = 6-4 = 2$ or $n = 2$. But when $n = 2$, $(a,b) \notin \B$. Therefore, (\ref{sub-B-x}) holds. We now show that (\ref{sub-B-y}) holds. Assume by contradiction that there exists $y > 0$ such that $(a, b-y) \in \B$. Then, by Lemma \ref{A-B1}, either $(a, b-y) = (2,2)$, or $(a, b-y) = (4,6)$. The first case cannot occur as $a = \lfloor \phi n \rfloor - 1$ with $n \neq 2$ and so $a \neq 2$. The second case implies either $a = 4$ or $a = 6$ and so either $\lfloor \phi n \rfloor = 5$ or $\lfloor \phi n \rfloor = 7$. However, there is no $n$ such that either $\lfloor \phi n \rfloor = 5$ or $\lfloor \phi n \rfloor = 7$ as $\lfloor \phi 3 \rfloor = 4$, $\lfloor \phi 4 \rfloor = 6$, and $\lfloor \phi 5 \rfloor = 8$. Therefore, (\ref{sub-B-y}) holds.

For (iii), let $p = (a,b) \notin \B \cup \P$ with $a \leq b$. If $a = 0$ then $b > 1$. Removing $b-1$ tokens leads $p$ to $(0,1) \in \B$. If $a = 1$, then one can remove the whole pile of size $b$. If $a = 2$, then $b > 2$ as $(2,2) \in \B$. Removing $b-2$ tokens from the larger pile leads $p$ to $(2,2) \in \B$. So we may suppose that $a \geq 3$. We can also assume that $a < b$ as if otherwise, one can move from $p$ to $(2,2) \in \B$ by removing $a-2$ tokens from both piles. By Lemma \ref{Comp}, there exists $n$ such that either $a = \lfloor \phi n \rfloor - 1$ or $a = \lfloor \phi n \rfloor + n - 1$.

Consider the case $a = \lfloor \phi n \rfloor - 1$. We have $b \neq \lfloor \phi n \rfloor + n - 1$. If $b = \lfloor \phi n \rfloor + n - 1 + i$ for some positive integer $i$, one can remove $i$ tokens from the pile of size $b$ leading $p$ to $(\lfloor \phi n \rfloor - 1, \lfloor \phi n \rfloor + n - 1) \in \B$. If $b = \lfloor \phi n \rfloor + n - 1-i$ for some $1 \leq i < n$ with $i\not=n-2$, one can remove $\lfloor \phi n \rfloor - \lfloor \phi (n-i) \rfloor$ tokens from both piles leading $p$ to $(\lfloor \phi (n-i) \rfloor - 1, \lfloor \phi (n-i) \rfloor +(n-i)-1)$, which belongs to $\B$ as $n-i\not=2$. If $b = \lfloor \phi n \rfloor + 1$, then $b=a+2$ and we have $a \geq 5$ as $(3,5) , (4,6) \in \P$. Then there is a move from $p$ to $(4,6)$ by removing $a-4$ coins from both piles.

Consider the case $a = \lfloor \phi n \rfloor + n - 1$. Note that $a \neq 3$. First assume that $a = 4$. If $b = 5$, removing 4 tokens from both piles leads $p$ to $(0,1) \in \B$. If $b > 5$ then $b > 6$ as $(4,6) \in \B$. One can remove $b-6$ tokens from the pile of size $b$ leading $p$ to $(4,6) \in \B$. Now assume that $a \geq 5$. Since $b > a$, $b = \lfloor \phi n \rfloor + n - 1 + i$ for some positive integer $i$. Removing $n+i$ tokens from the pile of size $b$ leads $p$ to $(\lfloor \phi n \rfloor - 1, \lfloor \phi n \rfloor + n - 1) \in \B$.
\end{proof}

In the remaining part of this section, we further  investigate the Sprague-Grundy function of $\R$-Wythoff. Let $a, c$ be nonnegative integers. The next theorem answers the question as to whether there exists $b$ such that $\G_{\R}(a,b) = c$.

%====================================================================================================================
\smallskip
\begin{theorem} \label{BW-Row}
Let $a, c$ be nonnegative integers. There exists an integer $b$ such that $\G_{\R}(a,b) = c$.
\end{theorem}

\begin{proof}
Note that the theorem holds for $a = 0$. Assume that $a \geq 1$. By Lemma \ref{Comp}, there exists $m$ such that either $a = \lfloor \phi m \rfloor$ or $a = \lfloor \phi m \rfloor + m$. The former case gives $\G_{\R}(a,\lfloor \phi m \rfloor + m) = 0$ and the latter case gives $\G_{\R}(a,\lfloor \phi m \rfloor) = 0$. Therefore, the theorem holds for $c = 0$. Assume that $c > 0$ and assume by contradiction that, for some $a\geq 0$,  the sequence $R_a = \{\G_{\R}(a,n)\}_{n \geq 0}$ does not contain $c$. We can assume that $c$ is the smallest integer not in the sequence $R_a$. Then there exists the smallest integer $b_0 \geq a$ such that
\begin{align} \label{BW-b0}
\{0, 1, \ldots,c -1\} \subseteq \{\G_{\R}(a,i) | i \leq b_0-1\}.
\end{align}
For each $s \geq 1$, let $b_s = b_0+s(a+1)$. We have
\[
\G_{\R}(a,b_s) = mex\{\G_{\R}(a,b_s-i), \G_{\R}(a-j,b_s-j) | 1 \leq i  \leq b_s, 1 \leq j \leq a\}.
\]
By (\ref{BW-b0}), the $mex$ set contains $\{0, 1, \ldots,c -1\}$. Note that $\G_{\R}(a,b_s) \neq c$. Therefore, $\G_{\R}(a,b_s) >c $ and so the $mex$ set contains $c$. Since $\G_{\R}(a,b_s-i) \neq c $ for all $i$, there exists some $j_s \leq a$ such that $\G_{\R}(a-j_s,b_s-j_s) = c $. Note that as $s$ varies, the integers $b_s$ assume infinitely many values, while $j_s\leq a$ for each $s$.
 So there must exist $s_1 < s_2$ such that $j_{s_1} = j_{s_2}$ and $\G_{\R}(a-j_{s_1},b_{s_1}-j_{s_1}) = \G_{\R}(a-j_{s_2},b_{s_2}-j_{s_2})$. This is impossible since one can move from $(a-j_{s_2},b_{s_2}-j_{s_2})$ to $(a-j_{s_1},b_{s_1}-j_{s_1})$ by removing $b_{s_2}-b_{s_1}$ tokens from the larger pile. Thus, the sequence $R_a$ contains $c$ and so $\G(a,b) = c$ for some $b$.
\end{proof}

%====================================================================================================================
\smallskip
\begin{theorem} \label{BW-Diagonal}
Let $a, c$ be nonnegative integers. There exists a unique $b$ such that $\G_{\R}(b,a+b) =c$.
\end{theorem}

\begin{proof}
The uniqueness holds as there exists a move from $(b_2,a+b_2)$ to $(b_1,a+b_1)$ if $b_1 < b_2$. Note that $\G_{\R}(\lfloor \phi a \rfloor,\lfloor \phi a \rfloor+a) = 0$ and so the theorem holds for $c = 0$. Assume that $c > 0$ and assume by contradiction that, for some $a \geq 0$, the sequence $\{\G_{\R}(n,a+n)\}_{n \geq 0}$ does not contain $c$. We can assume that $c$ is the smallest integer not in that sequence. Then there exists a smallest integer $b_0 > 0$ such that
\begin{align} \label{BW-b00}
\{0,1,\ldots, c-1\} \subseteq \{\G_{\R}(i,a+i) | i \leq b_0-1\}.
\end{align}

For each $s \leq b_0$, there exists at most one value $t_s \geq a+b_0$ such that $\G_{\R}(s,t_s) = c$. Let $S$ be the set of the values $t_s$, and set
\[
T_0 = \begin{cases}\max(S),&\ \text{if}\ S\not=\emptyset\\
a+b_0, &\ \text{otherwise}.
\end{cases}
\]
Then $\G_{\R}(s,t) \neq c$ for $s \leq b_0, t > T_0$.

We consider two possibilities for $a$. Assume that $a = 0$. Then $\G(i,i) \neq c$ for all $i$. Set $m = T_0+1$. We have
\[
\G_{\R}(m,m) = mex\{\G_{\R}(m,i), \G_{\R}(j,j) | i, j \leq m-1\}.
\]
By (\ref{BW-b00}), $\G_{\R}(m,m) \geq c$ and so $\G_{\R}(m,m) > c$ as $\G_{\R}(m,m) \neq c$. Since $\G_{\R}(j,j) \neq c$ for all $j$, there exists $i_0 \leq m-1$ such that $\G_{\R}(m,i_0) = c$. Note that $i_0 > b_0$ as otherwise $m \leq T_0$ giving a contradiction with $m = T_0+1$. We have
\[
\G_{\R}(i_0,i_0) = mex\{\G_{\R}(i_0,j), \G_{\R}(l,l) | j,l \leq i_0-1\}.
\]
By (\ref{BW-b00}), $\G_{\R}(i_0,i_0) \geq c$ and so $\G_{\R}(i_0,i_0) > c$ as $\G_{\R}(i_0,i_0) \neq c$. Since $\G_{\R}(l,l) \neq c$ for all $l$, there exists $j_0 < i_0$ such that $\G_{\R}(i_0,j_0) = c$. However, there exists a move from $(m,i_0)$ to $(i_0,j_0)$ as $m > i_0 > j_0$. This is a contradiction.

Assume that $a > 0$. For $k \in  \{1, 2, \ldots, 2a\}$, let $i_k = T_0+k$. We have
\begin{align*}
\G_{\R}&(i_k,a+i_k) =   \\
       &mex\{\G_{\R}(i,a+i), \G_{\R}(i_k,j) | i \leq i_k-1, j \leq i_k+a-1\}.
\end{align*}
By (\ref{BW-b00}), $\G_{\R}(i_k,a+i_k) \geq c$ and so $\G_{\R}(i_k,a+i_k) > c$ as $\G_{\R}(i_k,a+i_k) \neq c$. Therefore, there exists $j_k \leq a+i_k-1$ such that $\G_{\R}(i_k,j_k) = c$. Note that $j_k > b_0$ as otherwise $i_k \leq T_0$ giving a contradiction. We claim that $j_k > i_k-a$. Assume by contradiction that $j_k \leq i_k - a$ and so $a+j_k \leq i_k$. We have
\begin{align*}
\G_{\R}&(j_k,a+j_k) =   \\
       &mex\{\G_{\R}(i,a+i), \G_{\R}(j_k,l) | i \leq j_k-1, l \leq a+j_k-1\}.
\end{align*}
By (\ref{BW-b00}), $\G_{\R}(j_k,a+j_k) \geq c$ and so $\G_{\R}(j_k,a+j_k) > c$ as $\G_{\R}(j_k,a+j_k) \neq c$. Then, there exists $l_k \leq a+j_k-1$ such that $\G_{\R}(j_k,l_k) = c$. This is impossible since there is a move from $(j_k,i_k)$ to $(j_k,l_k)$ as $l_k < a+j_k \leq i_k$. Thus, for each $k$, there exists $j_k$ such that $\G_{\R}(i_k,j_k) = c$ and $i_k-a < j_k < i_k+a$. The last inequalities imply $-a < i_k - j_k < a$ and so there are at most $2a-1$ values $i_k - j_k$. However, there are $2a$ values $i_k$. It follows that there exist $k_1, k_2$ such that $i_{k_2} - j_{k_2} = i_{k_1} - j_{k_1}$ and so $i_{k_2} - i_{k_1} = j_{k_2} - j_{k_1}$. We can assume that $i_{k_1} < i_{k_2}$. Then there is a move from $(i_{k_2}, j_{k_2})$ to $(i_{k_1}, j_{k_1})$ by removing $(i_{k_2} - i_{k_1})$ tokens from both piles. This is a contradiction as these two positions have the same Sprague-Grundy value $c$.

Hence, the sequence $\{\G_{\R}(n,a+n)\}_{n \geq 0}$ contains $c$ and so $\G_{\R}(b,a+b)\} = c$ for some $b$.

\end{proof}

%====================================================================================================================
\smallskip
\begin{remark} \label{BW-Remark1}
Landman \cite{landman} stated that Theorem \ref{BW-Row} is true for Wythoff's game. To our knowledge, a proof for this statement has not appeared in the literature. The proof for Theorem \ref{BW-Row} can also be applied to Wythoff's game and $\E$-Wythoff (see Theorem \ref{CW-Row}). Blass and Fraenkel \cite{blass} also showed that Theorem \ref{BW-Diagonal} holds for Wythoff's game.
\end{remark}

Next, we give an upper bound and a lower bound for the Sprague-Grundy function. The following two lemmas can be proved by induction on $a$.

%====================================================================================================================
\smallskip
\begin{lemma} \label{BW-R0}
$\G_{\R}(0,a) = a$.
\end{lemma}

%====================================================================================================================
\smallskip
\begin{lemma} \label{BW-R12}
$\G_{\R}(1,a) = \G_{\R}(2,a) = a$ for $a \geq 3$.
\end{lemma}

%====================================================================================================================
\smallskip
\begin{lemma} \label{BW-R3}
For $a \geq 7$, we have
\begin{align*}
\G_{\R}(3,a) =
\begin{cases}
a,   &\text{ if $a \equiv 0,3 \pmod 4$};\\
a-4, &\text{ otherwise}.
\end{cases}
\end{align*}
\end{lemma}

\begin{proof}
First, it can be checked that $\G_{\R}(3,0) = \G_{\R}(3,1) = \G_{\R}(3,2) = 3$, $\G_{\R}(3,3) = 4$,
$\G_{\R}(3,4) = 2$, $\G_{\R}(3,5) = 0$, $\G_{\R}(3,6) = 1$, $\G_{\R}(3,7) = 7$. These values and Lemma \ref{BW-R12} give
\begin{align} \label{BW-G3a}
\G_{\R}(3,a) = mex\{0,1,2,3,4,\G_{\R}(3,i),a-3,a-2,a-1) | 7 \leq i \leq a-1 \}
\end{align}
for $a \geq 8$.

We will prove the lemma by induction on $a$. Note that the lemma holds for $a = 7$. One also can check that the lemma holds for $a = 8,9$. Assume that the lemma holds for $7 \leq a \leq n$ for some $n \geq 9$. We show that the lemma holds for $a = n+1$. Throughout this proof, for each $x$, we denote by $r_x$ the remainder $x - 4\lfloor x/4 \rfloor$, where $\lfloor . \rfloor$ is the integer part.

Assume that $r_a \in \{0,3\}$. When $7 \leq i \leq a-1$, by the inductive hypothesis, either $\G_{\R}(3,i) = i < a$ or $\G_{\R}(3,i) = i-4 < a$. By (\ref{BW-G3a}), it is sufficient to show that
\[\D = \{5,6,\ldots, a-4\} \subseteq \{\G_{\R}(3,i) | 7 \leq i \leq a-1\} = \F.\]
Let $m \in \D$. If $r_m \in \{0,3\}$ then $7 \leq m \leq a-4$ and so, by the inductive hypothesis, $\G_{\R}(3,m) = m$ implying $m \in \F$. If $r_m \in \{1,2\}$ then $m+4 \neq a$ as $r_a \in \{0,3\}$, so $m+4 < a$. Moreover, $m+4 \geq 9$. By the inductive hypothesis, $\G_{\R}(3,m+4) = m$ implying $m \in \F$.

Assume that $r_a \in \{1,2\}$. By (\ref{BW-G3a}), it is sufficient to show that
\begin{align*}
\begin{cases}
F = \{5,6,\ldots, a-5\} \subseteq \F,\\
a-4 \notin \F.
\end{cases}
\end{align*}
Let $m' \in \F$. If $r_{m'} \in \{0,3\}$, since $7 \leq m' \leq a-5$, we have $\G_{\R}(3,m') = m'$ by the inductive hypothesis and so $m' \in \F$. If $r_{m'} \in \{1,2\}$ then $r_{m'+4} \in \{1,2\}$. Since $9 \leq m'+4 \leq a-1$, we have $\G_{\R}(3,m'+4) = m'$ by the inductive hypothesis and so $m' \in \B$. We now show that $a-4 \notin \F$. Assume by contradiction that $a-4 \in \F$. There exists $i$ with $7 \leq i \leq a-1$ such that $\G_{\R}(3,i) = a-4$. If $r_i \in \{1,2\}$ then $\G_{\R}(3,i) = i-4$ by the inductive hypothesis and so $i=a$ giving a contradiction. If $r_i \in \{0,3\}$ then $\G_{\R}(3,i) = i$ by the inductive hypothesis and so $i = a-4$ implying $i \equiv a \pmod 4$ giving a contradiction.
\end{proof}

It has been shown that in Wythoff's game, $b-2a+1 \leq \G_{\R}(a,b) \leq a+b$ \cite{blass, landman}. For $\R$-Wythoff, we have the same lower bound and a stricter upper bound.

%====================================================================================================================
\smallskip
\begin{theorem} \label{BW-lower}
Let $a, b$ be positive integers with $4 \leq a \leq b$. Then $\G_{\R}(a,b) \geq b-2a+1$.
\end{theorem}

\begin{proof}
We prove the theorem by induction on $a+b$. Note that the theorem holds for $8 \leq a+b \leq 11$ as $b-2a+1 \leq 0$. Assume that the theorem holds for $a+b \leq n$ for some $n \geq 11$, we show that the theorem holds for $a+b = n+1$. Assume by contradiction that $c = \G_{\R}(a,b) < b-2a+1$. Then $b \geq 2a$.  Consider $\G_{\R}(a,b-1)$. By the inductive hypothesis, we have
\[\G_{\R}(a,b-1) \geq b-1-2a+1 = b-2a \geq c.\]
Since there is a move from $(a,b)$ to $(a,b-1)$, $\G_{\R}(a,b-1) \neq c$ and so $\G_{\R}(a,b-1) > c$. We have \begin{align} \label{BW-G(a,b-1)}
\G&_{\R}(a,b-1) =  \notag \\
  &mex\{\G_{\R}(a-i,b-1-i), \G_{\R}(a,b-1-j) | 1 \leq i \leq a, 1\leq j \leq b-1\}.
\end{align}
We claim that the values $\G_{\R}(a-i,b-1-i)$, where $1 \leq i \leq a$, are all larger than $c$. Note that $b-1-i \geq b-1-a \geq a-1 \geq 3$. Assume that $a-i \leq 2$. By Lemmas \ref{BW-R0} and \ref{BW-R12}, we have
\[\G_{\R}(a-i,b-1-i) = b-1-i \geq b-1-a > b-2a \geq c.\]
Assume that $a - i = 3$. Then $b-1-i = b-a+2 \geq a+2 \geq 6$. If $b-1-i = 6$ then $b = a+4$ which implies that $b-2a \leq 0$ and so $c = 0 < 1 = \G_{\R}(3,6)$. If $b-1-i \geq 7$, by Lemma \ref{BW-R3},
\[\G_{\R}(a-i,b-1-i) \geq b-1-i-4 = b-a-2 > b-2a \geq c.\]
Assume that $a-i \geq 4$. Since $b-1-i \geq 2a-1-i \geq 8$, by the inductive hypothesis, we have
\[\G_{\R}(a-i,b-1-i) \geq b-1-i - 2(a-i)+1 = b-2a+i > b-2a \geq c.\]
Now, since $\G_{\R}(a,b-1) > c$ and $\G_{\R}(a-i,b-1-i) > c$ for all $i$, it follows from (\ref{BW-G(a,b-1)}) that there exists some $j$ such that $\G_{\R}(a,b-1-j) = c$. This is impossible since $\G_{\R}(a,b) = c$ and there exists a move from $(a,b)$ to $(a,b-1-j)$. Therefore, $\G_{\R}(a,b) \geq b-2a+1$. This completes the proof.
\end{proof}

%====================================================================================================================
\smallskip
\begin{theorem} \label{BW-upper1}
Let $a,b$ be positive integers with $2 \leq a \leq b$. Then $\G_{\R}(a,b) \leq a+b-1$.
\end{theorem}

\begin{proof}
We argue by induction on $a+b$. First note that $\G_{\R}(2,2) = 1$ so the theorem holds for $a +b= 4$. Assume that the theorem holds for $a+ b \leq n$ for some $n \geq 4$. We show that the theorem holds for $a+b = n + 1$. We have
\begin{align*}
\G_{\R}(a,b) = mex \{\G_{\R}(a-i,b-i), \G_{\R}(a,b-j) | 1 \leq i \leq a, 1 \leq j \leq b\}.
\end{align*}
and so $\G_{\R}(a,b) \leq a+b-1$, provided we show that
\begin{align*}
\begin{cases}
\G_{\R}(a-i,b-i) < a+b-1, \\
\G_{\R}(a,b-j) < a+b-1.
\end{cases}
\end{align*}

Consider $\G_{\R}(a-i,b-i)$. If $a-i = 0$ then
\[\G_{\R}(a-i,b-i) = b-i < a+b-1.\]
If $a-i = 1$, as $\G_{\R}(1,1) = 2$, $\G_{\R}(1,2) = 0$ and $\G_{\R}(1,i) = i$ for $i \geq 3$ by Lemma \ref{BW-R12}, we have
\[\G_{\R}(a-i,b-i) < a+b-1.\]
If $a-i \geq 2$, by the inductive hypothesis, we have
\[\G_{\R}(a-i,b-i) \leq a-i + (b-i)-1 < a+b-1.\]

Consider $\G_{\R}(a,b-j)$.   If $b-j = 0$ then
\[\G_{\R}(a,b-j) = a < a+b-1.\]
If $b-j = 1$, as $\G_{\R}(2,1) = 0$ and $\G_{\R}(a,1)=a$ for $a \geq 3$ by Lemma \ref{BW-R12}, we have
\[\G_{\R}(a,b-j) < a+b-1.\]
If $b-j \geq 2$, by the inductive hypothesis, we have
\[\G_{\R}(a,b-j) \leq a+b-j-1 < a+b-1.\]
This completes the proof.
\end{proof}

%====================================================================================================================
\smallskip
We conclude this section by stating two conjectures.

\begin{conjecture} \label{BW-upper2}
Let $a, b$ be integers such that $4 \leq a \leq b$. Then
\begin{itemize}
\item $\G_{\R}(a,b) \leq b + \lfloor b/3 \rfloor - 1$ if $a < b$.
\item $\lfloor 3b/4 \rfloor \leq \G_{\R}(b,b) \leq b + \lfloor b/3 \rfloor$.
\end{itemize}
\end{conjecture}

Recall that a sequence $\{s_i\}_{i \geq i_0}$ is said to be \emph{additively periodic} if there exist $p \geq 1, n_0 \geq 0$ such that for all $n \geq n_0$, the condition $s_{n+p} = s_n + p$ holds. It is well known that in Wythoff's game, the sequence $\{\G(a,i)\}_{i\geq 0}$ is additively periodic for all $a$ \cite{Dress, landman}. We observe a similar behavior of Sprague-Grundy values of $\R$-Wythoff.

%====================================================================================================================
\smallskip
\begin{conjecture} \label{BW-additive}
Let $a$ be a nonnegative integer. The sequence $\{\G_{\R}(a,i)\}_{i\geq 0}$ is additively periodic.
\end{conjecture}

%====================================================================================================================
%====================================================================================================================
%====================================================================================================================
\medskip

\section{$\E$-Wythoff}
We first show that $\E$-Wythoff also preserves the $\P$-positions of Wythoff's game. We then give  formulas for those positions which have Sprague-Grundy value 1. We investigate the Sprague-Grundy function $\G_{\E}$ before giving some conjectures and a question at the end of the section.

%==================================================================================
\smallskip
\begin{theorem} \label{CW-P}
The $\P$-positions of $\E$-Wythoff are identical to those of Wythoff's game.
\end{theorem}

\begin{proof}
Recall that a move  can be added to the set of moves of Wythoff's game without changing the $\P$-positions if and only if that move cannot lead a $\P$-position to another $\P$-position in Wythoff's game \cite{blass}. Therefore, it suffices to show that for $l < k$, the move $\mathcal{M}$ removing $k$ tokens from the smaller pile and $l$ tokens from the other pile cannot lead a position of the form $(\lfloor \phi n \rfloor, \lfloor \phi n \rfloor +n)$ to another position of the form $(\lfloor \phi m \rfloor, \lfloor \phi m \rfloor +m)$. Assume by contradiction that there exist such positions $(\lfloor \phi n \rfloor, \lfloor \phi n \rfloor +n)$, $(\lfloor \phi m \rfloor, \lfloor \phi m \rfloor +m)$. We have then
\begin{align*}
\begin{cases}
\lfloor \phi n \rfloor - k = \lfloor \phi m \rfloor,\\
\lfloor \phi n \rfloor +n - l = \lfloor \phi m \rfloor +m.
\end{cases}
\end{align*}
The first equation implies that $n > m$. Replacing $\lfloor \phi n \rfloor$ by $\lfloor \phi m \rfloor +k$ in the second equation obtains $n+k-l=m$ implying $n \leq m$ giving a contradiction. Therefore, the move $\mathcal{M}$ cannot lead a $\P$-position to another $\P$-position in Wythoff's game.
\end{proof}

%==================================================================================

\begin{table}[ht]
\begin{center}
\begin{tabular}{c|cccccccccc}
9     &9 &10&11&12&2   &1  &15  &16 &17  &18 \\
8     &8 &6 &9  &10 &11 &13&14  &15 &16  &17 \\
7     &7 &8 &4  &2  &0   &12 &13 &14  &15  &16 \\
6     &6 &7 &8  &1  &10 &11 &12 &13  &14  &15 \\
5     &5 &3 &6  &0  &9   &10 &11 &12  &13  &1 \\
4     &4 &5 &1  &7  &8   &9  &10  &0   &11   &2\\
3     &3 &4 &5  &6  &7   &0  &1   &2    &10   &12 \\
2     &2 &0 &3  &5  &1   &6  &8   &4    &9     &11 \\
1     &1 &2 &0  &4  &5   &3  &7   &8    &6     &10 \\
0     &0 &1 &2  &3  &4   &5  &6   &7    &8     &9 \\
\hline
a/b &0&1&2&3&4&5&6&7&8&9
 \end{tabular}
\caption{Sprague-Grundy values $\G_{\E}(a,b)$ for $a,b\leq 9$}\label{T2}
\end{center}
\end{table}

Table \ref{T2} gives the  Sprague-Grundy values of position $(a,b)$ for $a,b\leq 9$. We now give formulas for those positions which have Sprague-Grundy value 1.

%===================================================================================
\smallskip
\begin{theorem} \label{CW-V1}
In $\E$-Wythoff, the position $(a,b)$ with $a \leq b$ has Sprague-Grundy value 1 if and only if $(a,b)$ is of the form
\[(\lfloor \phi n \rfloor - 1, \lfloor \phi n \rfloor +n-1)\]
for some $n \geq 1$.
\end{theorem}

\begin{proof}
Let
\[\Q = \{(\lfloor \phi n \rfloor - 1, \lfloor \phi n \rfloor +n-1) | n \geq 1\}.\]
By Theorem \ref{CW-P}, the set of $\P$-positions of the game is
\[\P = \{(\lfloor \phi m \rfloor, \lfloor \phi m \rfloor + m)| m \geq 0\}.\]
We need to prove that
\begin{itemize}
\item [(i)] $\Q \cap \P = \emptyset$,
\item [(ii)] There is no move from a position in $\Q$ to a position in $\Q$,
\item [(iii)] From every position not in $M\cup \P$, there is one move to some position in $\Q$.
\end{itemize}

Note that (i) follows Lemma \ref{A-B} and (ii) follows Lemma \ref{A-B1}. For (iii), let $p = (a,b) \notin \Q \cup \P$ with $a \leq b$. We can assume that $a < b$ since if otherwise, one can move from $p$ to $(0,1) \in \Q$ by removing $a$ tokens from one pile and $a-1$ tokens from the other pile. By Lemma \ref{Comp}, there exists $n$ such that either $a = \lfloor \phi n \rfloor - 1$ or $a = \lfloor \phi n \rfloor + n - 1$. If the former case occurs, we have $b \neq \lfloor \phi n \rfloor + n - 1$. If $b = \lfloor \phi n \rfloor + n - 1 + i$, removing $i$ tokens from the pile of size $b$ leads $p$ to $(\lfloor \phi n \rfloor - 1, \lfloor \phi n \rfloor + n - 1) \in \Q$. If $b = \lfloor \phi n \rfloor + n - 1-i$, then $m = n-i > 0$ as $b > a$. One can remove $\lfloor \phi n \rfloor - \lfloor \phi m \rfloor$ tokens from both piles leading $p$ to $(\lfloor \phi m \rfloor - 1, \lfloor \phi m \rfloor +m-1) \in \Q$. If the latter case occurs, since $b > a$, $b = \lfloor \phi n \rfloor + n - 1 +j$ for some $j > 0$. Removing $n+j$ tokens from the pile of size $b$ leads $p$ to $(\lfloor \phi n \rfloor - 1, \lfloor \phi n \rfloor + n - 1) \in \Q$.
\end{proof}

Theorem \ref{CW-V1} raises two interesting things. Firstly, there is a nice connection between these positions and the $\P$-positions. A position $(a,b)$ has Sprague-Grundy value 1 if and only if $(a+1,b+1)$ is a $\P$-position. Secondly, the sets of positions which have Sprague-Grundy value 1 in $\R$-Wythoff and in $\E$-Wythoff have only three different positions $(2,2), (2,4), (4,6)$. Is this a fluke? The reasons for this similarity need further investigation (see Question \ref{Q1}).

%====================================================================================
\smallskip
\begin{theorem} \label{CW-Row}\
Let $a,c$ be nonnegative integers. There exists a unique $b$ such that $\G_{\E}(a,b) = c$.
\end{theorem}

\begin{proof}
The uniqueness holds as one can move from $(a,b+i)$ to $(a,b)$. The existence is established by the same argument used in the proof of Theorem \ref{BW-Row}.
\end{proof}

%====================================================================================================================
We now give the lower and upper bounds for the Sprague-Grundy function of $\E$-Wythoff. We first recall a simple result in Wythoff's game.

%=======================================================================
\smallskip
\begin{lemma} \cite{blass} \label{CW-R1}
In Wythoff's game, for $a \geq 0$, we have
\begin{align*}
\G(1,a) =
\begin{cases}
a+1, &\text{if $a \equiv 0,1 \pmod 3$};\\
a-2, &\text{otherwise}.
\end{cases}
\end{align*}
\end{lemma}

%============================================================================
\smallskip
\begin{remark} \label{CW-Remark1}
Lemma \ref{CW-R1} is also true for the Sprague-Grundy function of $\E$-Wythoff as both Wythoff's game and $\E$-Wythoff allow the same moves from the position $(1, a)$ and all positions reached from $(1, a)$.
\end{remark}

%========================================================================
\smallskip
\begin{lemma} \label{CW-R2}
In $\E$-Wythoff, for $a \geq 0, a \neq 1$, we have
\begin{align*}
\G_{\E}(2,a) =
\begin{cases}
a+2, &\text{if $a \equiv 0 \pmod 3$};\\
a-3, &\text{if $a \equiv 1 \pmod 3$};\\
a+1, &\text{if $a \equiv 2 \pmod 3$}.
\end{cases}
\end{align*}
\end{lemma}

\begin{proof}
We prove the lemma by induction on $a$. Throughout this proof, we denote by $\equiv $ the congruence modulo 3. Calculations show that the lemma holds for $a \leq 8$. Assume that the lemma holds for $a \leq n$ for some $n \geq 8$, we show that the lemma holds for $a = n+1$. We have
\begin{align} \label{CW-G2a}
\G_{\E}&(2,a) \notag \\
       &= mex\{\G_{\E}(2-i,a-j), \G_{\E}(2,a-k) | 0 \leq j \leq i \leq 2, i \geq 1, 1 \leq k \leq a\}  \notag \\
       &= mex \{a,a-1,a-2,\G_{\E}(1,a), \G_{\E}(1,a-1), \G_{\E}(2,i) | i \leq a-1\}.
\end{align}
Set $S = \{\G_{\E}(2,i) | i \leq a-1\}$.

If $a \equiv  0$, by Remark \ref{CW-Remark1}, $\G_{\E}(1,a) = a+1$, $\G_{\E}(1,a-1)=a-3$ and so (\ref{CW-G2a}) becomes
\[\G_{\E}(2,a) = mex \{a+1,a,a-1,a-2,a-3, \G_{\E}(2,i) | i \leq a-1\}.\]
It remains to show that $\{0, 1, \ldots, a-4\} \subseteq S$ and $a+2 \notin S$. For the first condition, note that $0 = \G_{\E}(2,1) \in S$. Let $m \in \{1,2, \ldots, a-4\}$. Then $m+3 \leq a-1$. If $m \equiv  0$ then $m-1 \equiv  2$. By the inductive hypothesis, we have $\G_{\E}(2,m-1) = m$ and so $m \in S$. If $m \equiv  1$ then $m+3 \equiv  1$. By the inductive hypothesis, we have $\G_{\E}(2,m+3) = m$ and so $m \in S$. If $m \equiv  2$ then $m-2 \equiv  0$. By the inductive hypothesis, we have $\G_{\E}(2,m-2) = m$ and so $m \in S$. For the second condition, note that $\G_{\E}(2,0) = 2$, $\G_{\E}(2,1) = 0$. Moreover, for $3 \leq i \leq a-1$, by the inductive hypothesis, we have $\G_{\E}(2,i) \leq i+2 \leq a+1$. Therefore, $\G_{\E}(2,i) < a+2$ for all $i \leq a-1 $ and so $a+2 \notin S$.

If $a \equiv  1$, by Remark \ref{CW-Remark1}, $\G_{\E}(1,a) = a+1$, $\G_{\E}(1,a-1)=a$ and so (\ref{CW-G2a}) becomes
\[\G_{\E}(2,a) = mex \{a+1,a,a-1,a-2, \G_{\E}(2,i) | i \leq a-1\}.\]
It remains to show that $\{0, 1, \ldots, a-4\} \subseteq S$ and $a-3 \notin S$. For the first condition, let $a' = a-1$. Then $a' \equiv  0$. The case $a \equiv  0$ above gives
\[\{0,1, \ldots, a-5\} = \{0,1, \ldots, a'-4\} \subseteq \{\G_{\E}(2,i) | i \leq a'-1\} \subseteq S.\]
We now need to show that $a-4 \in S$. Note that $a-5 \equiv  2$. By the inductive hypothesis, we have $\G_{\E}(2,a-5) = a-4$ and so $a-4 \in S$. For the second condition, assume by contradiction that $a-3 \in S$. Then, there exists $i \leq a-1$ such that $\G_{\E}(2,i) = a-3$. If $i \equiv  0$ then $\G_{\E}(2,i) = i+2$ by the inductive hypothesis. We then have $i+2 = a-3$ and so $a \equiv  i+2 \equiv  2$ giving a contradiction. If $i \equiv  1$, then $i \neq 1$ as $\G_{\E}(2,1) = 0 < a-3$. By the inductive hypothesis, we have $\G_{\E}(2,i) = i-3$ which implies $i-3 = a-3$ and so $i = a$ giving a contradiction. If $i \equiv  2$ then $\G_{\E}(2,i) = i+1$ by the inductive hypothesis. It follows that $i+1 = a-3$ and so $a \equiv  i+1 \equiv  0$ giving a contradiction. Thus, the second condition holds.

If $a \equiv  2$, Remark \ref{CW-Remark1} gives $\G_{\E}(1,a) = a-2$, $\G_{\E}(1,a-1)=a$ and so (\ref{CW-G2a}) becomes
\[\G_{\E}(2,a) = mex \{a,a-1,a-2, \G_{\E}(2,i) | i \leq a-1\}.\]
It remains to show that $\{0, 1, \ldots, a-3\} \subseteq S$ and $a+1 \notin S$. For the first condition, let $a'' = a-1$. Then $a'' \equiv  1$. From the case $a \equiv  1$ above, we have
\[\{0,1, \ldots, a-5\} = \{0,1, \ldots, a''-4\} \subseteq \{\G_{\E}(2,i) | i \leq a''-1\} \subseteq S.\]
We now need to show that $a-4 \in S, a-3 \in S$. By the inductive hypothesis, we have $\G_{\E}(2,a-1) = a-4$ and so $a-4 \in S$. Also by the inductive hypothesis, we have $\G_{\E}(2,a-5) = a-3$ and so $a-3 \in S$. For the second condition, assume by contradiction that $a+1 \in S$. By the definition of $S$, there exists $i \leq a-1$ such that $\G_{\E}(2,i) = a+1$. This condition holds if and only if $i = a-1 \equiv  1$ as $\G_{\E}(2,1) = 0$ and $\G_{\E}(2,i) \leq i+2 \leq a+1$ for $i > 1$ by the inductive hypothesis. However, by the inductive hypothesis we have  $\G_{\E}(2,i) = i-3 = a-4$, giving a contradiction. Therefore, $a+1 \notin S$
\end{proof}

%====================================================================================
\smallskip
The proof of the following result is essentially the same as that of Theorem \ref{BW-upper1}, with Lemma \ref{CW-R2} replacing Lemma \ref{BW-R12}. We leave the details to the reader.

\begin{theorem} \label{CW-low}
Let $3 \leq a \leq b$, we have $\G_{\E}(a,b) \geq b - 2a + 1$.
\end{theorem}

%===================================================================================
\smallskip
\begin{theorem} \label{CW-high}
In $\E$-Wythoff, for $a \leq b$, we have $\G_{\E}(a,b) \leq a+b$.
\end{theorem}

\begin{proof}
The proof is by induction on $a+b$. Note that the theorem holds for $a + b \leq 5$ by Remark \ref{CW-Remark1} and Lemma \ref{CW-R2}. In general, one has
\begin{align*}
\G_{\E}&(a,b) = \\
       &mex\{ \G_{\E}(a,b-i), \G_{\E}(a-k,b-l) | 1\leq i \leq b, 0 \leq l \leq k \leq a, k \geq 1\}.
\end{align*}
By the inductive hypothesis, all elements in the $mex$ set are less than $a+b$ and so $\G_{\E}(a,b) \leq a+b$. \end{proof}

%==================================================================================
\smallskip
\begin{remark} \label{CW-Remark2}
Let $K, L$ be sets of positive integers. Consider the extension of Wythoff's game obtained by adjoining a move which removes $k \in K$ tokens from the smaller pile (or any pile if the two piles have the same size) and $l \in L$ tokens from the other pile such that $k > l$ and $k,l$ satisfy some given relation $R(k,l)$. An example for $R(k,l)$ is $k = l+1$. Then, one can check that Theorems \ref{CW-P}, \ref{CW-Row}, \ref{CW-low}, and \ref{CW-high} still hold for this extension without changing the proofs.
\end{remark}

We conclude by stating conjectures and a question concerning $\E$-Wythoff. The first two following conjectures describe the distribution of Sprague-Grundy values on diagonals parallel to the main diagonal and the next conjecture describes the distribution of Sprague-Grundy values on each row of the expansion of Table \ref{T2}.

%==================================================================================
\smallskip
\begin{conjecture} \label{EW-C1}
Let $r \geq 0, a \geq 2r$. Then
\begin{align*}
\G_{\E}(a,a+r) =
\begin{cases}
3,    &\text{if $a = 2, r = 0$}; \\
2a+r, &\text{otherwise}.
\end{cases}
\end{align*}
\end{conjecture}

%===================================================================================
\smallskip
\begin{conjecture}
Let $a \geq 4, 2 \leq r \leq a+1$. Then $\G_{\E}(a,3a+r) = 4a+r-1$.
\end{conjecture}

%===================================================================================
\smallskip
\begin{conjecture} \label{CW-additive}
Let $a$ be nonnegative integer. The sequence $\{\G_{\E}(a,i)\}_{i\geq 0}$ is additively periodic.
\end{conjecture}

%===================================================================================
\smallskip
\begin{question} \label{Q1}
Does there exist another variant of Wythoff's game preserving its $\P$-position and accepting the formula $(\lfloor \phi n \rfloor - 1, \lfloor \phi n \rfloor +n-1)$ for all but possibly a finite number of positions which have Sprague-Grundy value 1?
\end{question}

%===================================================================================
%===================================================================================
%==================================================================================
%==================================================================================
\smallskip
\begin{ack}
I thank my supervisor Grant Cairns for several valuable suggestions regarding content and exposition. I thank the referees for their helpful comments which improved the presentation of the paper.
\end{ack}

%==================================================================================
%==================================================================================
%==================================================================================


\small

\begin{thebibliography}{19}

\bibitem{beatty1} S. Beatty, A. Ostrowski, J. Hyslop, A.C. Aitken, Solution to problem 3173, Amer. Math. Monthly 34 (1927) 159--160.

\bibitem{blass} U. Blass, A.S. Fraenkel, The {S}prague-{G}rundy function for {W}ythoff's game, Theoret. Comput. Sci. 75 (1990) 311--333.

\bibitem{MEuclid} G. Cairns, N. B. Ho, A restriction of the game Euclid, Submitted, arXiv:1202.4597v1.

\bibitem{Min} G. Cairns, N. B. Ho, Min, a combinatorial game having a connection with prime numbers, Integers 10 (2010), G03, 765--770.

\bibitem{CHL} G. Cairns, N. B. Ho, T. Lengyel, The Sprague-Grundy function of the real game Euclid, Discrete Math. 311 (2011), 457--462.

\bibitem{Gen-Connell} I.G. Connell, A generalization of {W}ythoff's game, Canad. Math. Bull. 2 (1959) 181--190.

\bibitem{Dress} A. Dress, A. Flammenkamp, N. Pink, Additive periodicity of the {S}prague-{G}rundy function of certain {N}im games, Adv. in Appl. Math. 22 (1999) 249--270.

\bibitem{Nim-Wythoff} E. Duch{\^e}ne, A.S. Fraenkel, S. Gravier, R.J. Nowakowski, Another bridge between NIM and WYTHOFF, Australas. J. Combin. 44 (2009) 43--56.

\bibitem{Ext-Res} E. Duch{\^e}ne,  A.S. Fraenkel, R.J. Nowakowski, M. Rigo, Extensions and restrictions of {W}ythoff's game preserving its $\P$-positions, J. Combin. Theory Ser. A 117 (2010) 545--567.

\bibitem{Geo-ext} E. Duch{\^e}ne, S. Gravier, Geometrical extensions of Wythoff's game, Discrete Math. 309 (2009) 3595--3608.

\bibitem{inv} E. Duch{\^e}ne, M. Rigo, Invariant games, Theoret. Comput. Sci. 411 (2010) 3169–-3180

\bibitem{Heapgame} A.S. Fraenkel, Heap games, numeration systems and sequences, Ann. Comb. 2 (1998) 197--210.

\bibitem{Howtobeat} A.S. Fraenkel, How to beat your Wythoff games' opponent on three fronts, Amer. Math. Monthly 89 (1982)  353--361.

\bibitem{Gen-Fra} A.S. Fraenkel, I. Borosh, A generalization of Wythoff's game, J. Combinatorial Theory Ser. A 15 (1973) 175--191.

\bibitem{Nimhoff} A.S. Fraenkel, M. Lorberbom, Nimhoff games, J. Combin. Theory Ser. A 58 (1991) 1--25.

\bibitem{Adjoining} A.S. Fraenkel, M. Ozery, Adjoining to Wythoff's game its $\P$-positions as moves, Theoret. Comput. Sci. 205 (1998) 283--296.

\bibitem{End-Wyt} A.S. Fraenkel, E. Reisner, The game of End-Wythoff, in: M.H. Albert, R.J. Nowakowski (Eds.), Games of No Chance III, Cambridge University Press, Cambridge, 2009, pp. 329--347.

\bibitem{anew} A.S. Fraenkel, D. Zusman, A new heap game, Theoret. Comput. Sci. 252 (2001) 5--12.

\bibitem{Ho} N.B. Ho, Variants of Wythoff's game translating its $\P$-positions, Preprint.

\bibitem{hog} V.E. Hoggatt Jr., M. Bicknell-Johnson, R. Sarsfield, A generalization of Wythoff's game, Fibonacci Quart. 17 (1979) 198--211.

\bibitem{Some-Gen} J.C. Holladay, Some Generalizations of {W}ythoff's Game and Other Related Games, Math. Mag. 41 (1968)  7--13.

\bibitem{landman} H. Landman, A simple FSM-based proof of the additive periodicity of the {S}prague-{G}rundy function
of Wythoff's game, in: R.J. Nowakowski (Ed.), More Games of No Chance, Cambridge University Press,
Cambridge, 2002, pp. 383–-386.

\bibitem{nivasch} G. Nivasch, More on the {S}prague-{G}rundy function for Wythoff's game, in: M.H. Albert, R.J. Nowakowski (Eds.), Games of No Chance III, Cambridge University Press, Cambridge, 2009, pp. 377–-410.

\bibitem{Wyt} W.A. Wythoff, A modification of the game of {N}im, Nieuw Arch. Wiskd. 7 (1907) 199--202.

 \end{thebibliography}
\end{document}